\documentclass[12pt]{article}
\usepackage{graphicx}
\usepackage{amsmath,amsthm,amssymb,enumerate}
\usepackage{euscript,mathrsfs}
\usepackage[left=2cm,right=2cm,top=3.5cm,bottom=3.5cm]{geometry}
\usepackage{color}

\usepackage{soul}

\catcode`\@=11 \@addtoreset{equation}{section}

\catcode`\@=12

\allowdisplaybreaks

\newtheorem{Theorem}{Theorem}[section]
\newtheorem{Proposition}[Theorem]{Proposition}
\newtheorem{Lemma}[Theorem]{Lemma}
\newtheorem{Corollary}[Theorem]{Corollary}

\theoremstyle{definition}
\newtheorem{Definition}[Theorem]{Definition}

\newtheorem{Remark}[Theorem]{Remark}

\newcommand{\bTheorem}[1]{
\begin{Theorem} \label{T#1} }
\newcommand{\eT}{\end{Theorem}}

\newcommand{\bProposition}[1]{
\begin{Proposition} \label{P#1}}
\newcommand{\eP}{\end{Proposition}}

\newcommand{\bLemma}[1]{
\begin{Lemma} \label{L#1} }
\newcommand{\eL}{\end{Lemma}}

\newcommand{\bCorollary}[1]{
\begin{Corollary} \label{C#1} }
\newcommand{\eC}{\end{Corollary}}

\newcommand{\bRemark}[1]{
\begin{Remark} \label{R#1} }
\newcommand{\eR}{\end{Remark}}

\newcommand{\bDefinition}[1]{
\begin{Definition} \label{D#1} }
\newcommand{\eD}{\end{Definition}}

\newcommand{\Ds}{\mathbb{D}_x}

\newcommand{\bfphi}{\boldsymbol{\varphi}}

\newcommand{\bFormula}[1]{
\begin{equation} \label{#1}}
\newcommand{\eF}{\end{equation}}

\newcommand{\Ov}[1]{\overline{#1}}

\newcommand{\aleq}{\stackrel{<}{\sim}}

\newcommand{\ageq}{\stackrel{>}{\sim}}

\newcommand{\vr}{\varrho}

\newcommand{\tvr}{\tilde \vr}
\newcommand{\tvu}{{\tilde \vu}}

\newcommand{\vu}{\vc{u}}
\newcommand{\vm}{\vc{m}}

\newcommand{\vc}[1]{{\bf #1}}

\newcommand{\Div}{{\rm div}_x}
\newcommand{\Grad}{\nabla_x}

\newcommand{\dx}{\,{\rm d} {x}}

\newcommand{\dt}{\,{\rm d} t }

\newcommand{\intO}[1]{\int_{\Omega} #1 \ \dx}

\newcommand{\D}{{\rm d}}

\newcommand{\ep}{\varepsilon}

\newcommand{\R}{\mathbb{R}}

\def\softd{{\leavevmode\setbox1=\hbox{d}%
          \hbox to 1.05\wd1{d\kern-0.4ex{\char039}\hss}}}
\definecolor{Cgrey}{rgb}{0.85,0.85,0.85}
\definecolor{Cblue}{rgb}{0.50,0.85,0.85}
\definecolor{Cred}{rgb}{1,0,0}
\definecolor{fancy}{rgb}{0.10,0.85,0.10}

\newcommand\Cbox[2]{%
    \newbox\contentbox%
    \newbox\bkgdbox%
    \setbox\contentbox\hbox to \hsize{%
        \vtop{
            \kern\columnsep
            \hbox to \hsize{%
                \kern\columnsep%
                \advance\hsize by -2\columnsep%
                \setlength{\textwidth}{\hsize}%
                \vbox{
                    \parskip=\baselineskip
                    \parindent=0bp
                    #2
                }%
                \kern\columnsep%
            }%
            \kern\columnsep%
        }%
    }%
    \setbox\bkgdbox\vbox{
        \color{#1}
        \hrule width  \wd\contentbox %
               height \ht\contentbox %
               depth  \dp\contentbox
        \color{black}
    }%
    \wd\bkgdbox=0bp%
    \vbox{\hbox to \hsize{\box\bkgdbox\box\contentbox}}%
    \vskip\baselineskip%
}

\date{}

\begin{document}


\title{On the long--time behavior of dissipative solutions to models of non-Newtonian compressible fluids}

\author{Eduard Feireisl
\thanks{The work of E.F. was partially supported by the
Czech Sciences Foundation (GA\v CR), Grant Agreement
18--05974S. The Institute of Mathematics of the Academy of Sciences of
the Czech Republic is supported by RVO:67985840.} \and Young--Sam Kwon
\thanks{The work of Y.--S. K. was partially supported by
the National Research Foundation of Korea (NRF-2017R1D1A1B03030249  and NRF-2019H1D3A2A01101128) }
\and Anton\' \i n Novotn\' y { \thanks{The work of A.N. was supported by Brain Pool program funded by the Ministry of Science and ICT through the National Research Foundation of Korea (NRF-2019H1D3A2A01101128).}}
}


\maketitle

\centerline{Institute of Mathematics of the Academy of Sciences of the Czech Republic;}
\centerline{\v Zitn\' a 25, CZ-115 67 Praha 1, Czech Republic}

\centerline{Institute of Mathematics, Technische Universit\"{a}t Berlin,}
\centerline{Stra{\ss}e des 17. Juni 136, 10623 Berlin, Germany}
\centerline{feireisl@math.cas.cz}

\centerline{and}

\centerline{Department of Mathematics, Dong-A University}
\centerline{Busan 49315, Republic of Korea}
\centerline{ykwon@dau.ac.kr}

\centerline{and}

\centerline{IMATH, EA 2134, Universit\'e de Toulon,}
\centerline{BP 20132, 83957 La Garde, France}
\centerline{novotny@univ-tln.fr}

\begin{abstract}

We identify a class \emph{maximal} dissipative solutions to models of compressible viscous fluids that maximize the energy dissipation rate. Then we show that any maximal dissipative solution approaches an equilibrium state for large times.

\end{abstract}


\bigskip


\section{Problem formulation}
\label{P}

We consider a mathematical model of a compressible viscous fluid occupying a bounded physical domain
$\Omega \subset \R^d$, $d=2,3$. The state of the fluid at a given time $t \geq 0$ and a spatial position $x \in \Omega$
is characterized by the mass density $\vr = \vr(t,x)$ and the bulk velocity $\vu = \vu(t,x)$ satisfying the following
system of partial differential equations:

\begin{equation} \label{P1}
\begin{split}
\partial_t \vr + \Div (\vr \vu) &= 0, \\
\partial_t (\vr \vu) + \Div (\vr \vu \otimes \vu) + \Grad p(\vr) &= \Div \mathbb{S},
\end{split}
\end{equation}
where $p$ is the pressure and $\mathbb{S}$ the viscous stress tensor. The viscous stress is related to the
symmetric velocity gradient
\[
\Ds \vu = \frac{\Grad \vu + \Grad \vu^t}{2}
\]
through a general \emph{rheological law}
\begin{equation} \label{P2}
\mathbb{S} \in \partial F(\Ds),
\end{equation}
where $\partial F$ is the subdifferential of a convex potential $F$. In view of Fenchel--Young identity,
the relation \eqref{P2} can be written in an ``implicit'' form
\begin{equation} \label{P3}
\mathbb{S}: \Ds \vu =
F(\Ds \vu) + F^*(\mathbb{S}),
\end{equation}
where $F^*$ is the conjugate of $F$. Finally, we consider the
no--slip boundary conditions
\begin{equation} \label{P4}
\vu|_{\partial \Omega} = 0,
\end{equation}
together with the initial conditions
\begin{equation} \label{P5}
\vr(0, \cdot) = \vr_0, \ \vr \vu(0, \cdot) = \vm_0.
\end{equation}

Smooth solutions of \eqref{P1}--\eqref{P5} satisfy the total energy balance
\begin{equation} \label{P6}
\intO{ E(\vr, \vm)(\tau, \cdot) } + \int_0^\tau \intO{ \left( F(\Ds \vu) + F^*(\mathbb{S}) \right) } \dt =
\intO{ E(\vr_0, \vm_0) }
\end{equation}
for any $\tau \geq 0$, where $E$ is the total energy,
\[
E = \frac{1}{2} \frac{|\vm|^2}{\vr} + P(\vr),\
\vm \equiv \vr \vu,\ P'(\vr) \vr - P(\vr) = p(\vr).
\]
In addition, in view of \eqref{P1}, \eqref{P4}, the total mass of the fluid is a conserved quantity,
\begin{equation} \label{P7}
M = \intO{ \vr(\tau, \cdot) } = \intO{ \vr_0 }
\end{equation}
for any $\tau \geq 0$. In accordance with the Second law of thermodynamics, the dissipation potential $F$ must satisfy
\[
F(\Ds \vu) + F^*(\mathbb{S}) \geq 0;
\]
whence the equilibrium (time independent) states $[\tvr, \tvu]$ satisfy
\begin{equation} \label{P8}
F(\Ds \tvu) + F^*(\widetilde{\mathbb{S}}) = 0, \ \tvu|_{\partial \Omega} = 0,\
\intO{ \tvr } = M.
\end{equation}
For \emph{real} fluids, the dissipation is always present therefore \eqref{P8} implies
\[
\tvu = 0,\ \widetilde{\mathbb{S}} = 0;
\]
whence, in accordance with \eqref{P1}
\[
\partial_t \tvr = 0,\ \Grad p(\tvr) = 0.
\]
Thus if the pressure is a strictly monotone (increasing) function of $\vr$, we may infer that
\begin{equation} \label{P9}
\tvr(x) = \Ov{\vr},\ \mbox{where}\ \Ov{\vr} > 0 \ \mbox{is a constant},\ \Ov{\vr} |\Omega| = M.
\end{equation}
In view of \eqref{P9}, it is convenient to fix the pressure potential $P$ in the energy,
\[
E\left(\vr, \vm \Big| \Ov{\vr} \right) =  \frac{1}{2} \frac{|\vm|^2}{\vr} + P(\vr) - P'(\Ov{\vr})(\vr - \Ov{\vr}) -
P(\Ov{\vr}).
\]
The energy being a convex function of $[\vr, \vm]$, the quantity $E\left(\vr, \vm \Big| \Ov{\vr} \right)$ can be interpreted as
the Bregman distance between $[\vr, \vm]$ and the equilibrium state $[\Ov{\vr}, 0]$.

Our goal is to study the long--time behavior of solutions to the problem \eqref{P1}--\eqref{P5}, in particular, we show that any individual trajectory approaches a single equilibrium determined uniquely by the total mass of the fluid. To the best of our knowledge,
the problem of \emph{global existence} for the problem \eqref{P1}--\eqref{P5} is largely open even in the class of weak 
(distributional) solutions; the only exception being the Navier--Stokes system, where both $F$ and $F^*$ are quadratic, and
the global existence of weak solutions was shown by Lions
\cite{LI4} and extended in \cite{FNP}, and the problem with linear pressure and exponentially growing
viscosity coefficients studied by Mamontov \cite{MAM1}, \cite{MAM2}.

In the light of the afore mentioned difficulties with global solvability, 
we consider the problem \eqref{P1}--\eqref{P5} in the framework
of \emph{dissipative solutions} introduced in \cite{AbbFeiNov}.
The leading idea is to replace the viscous stress $\mathbb{S}$ by $\mathbb{S}_{\rm eff} = \mathbb{S} -
\mathfrak{R}$, with an extra stress $\mathfrak{R}$ called Reynolds stress,
\[
\mathfrak{R}(\tau) \in \mathcal{M}^+(\Ov{\Omega}; \R^{d \times d}_{\rm sym}), \ \tau > 0,
\]
where $\mathcal{M}^+(\Ov{\Omega}; \R^{d \times d}_{\rm sym})$ is the set of positively semi--definite matrix--valued measures
in $\Ov{\Omega}$.

The dissipative solutions satisfy
\begin{equation} \label{P10}
\begin{split}
\partial_t \vr + \Div (\vr \vu) &= 0, \\
\partial_t (\vr \vu) + \Div (\vr \vu \otimes \vu) + \Grad p(\vr) &= \Div \mathbb{S}_{\rm eff},
\end{split}
\end{equation}
together with the energy inequality
\begin{equation} \label{P11}
\begin{split}
\frac{{\rm d}}{{\rm d}t} \left[ \intO{ E\left( \vr, \vm \Big| \Ov{\vr} \right) (\tau, \cdot) } + D \int_{\Ov{\Omega}} \D\ {\rm tr} [\mathfrak{R}](\tau)\right]  &+
\intO{ \Big( F(\Ds \vu) + F^* (\mathbb{S}_{\rm eff} + \mathfrak{R}) \Big) } \dt \leq 0,\\
\left[ \intO{ E\left( \vr, \vm \Big| \Ov{\vr} \right)  } + D \int_{\Ov{\Omega}} \D\ {\rm tr} [\mathfrak{R}] \right]\
&\leq \intO{ E\left( \vr_0, \vm_0 \Big| \Ov{\vr} \right) },
\end{split}
\end{equation}
where $D > 0$ is a constant determined solely by the structural properties of $F$ and $p$, see Section \ref{W} for details.

Although the class of dissipative solutions is apparently larger than that of conventional weak (distributional) solutions,
they still enjoy the following properties:

\begin{itemize}
\item {\bf Existence.} The dissipative solutions exist globally--in--time for any finite energy initial data,
\[
\vr \in C_{{\rm weak, loc}}([0, \infty); L^\gamma(\Omega)),\
\vm \equiv \vr \vu \in C_{{\rm weak,loc}}([0, \infty); L^{\frac{2 \gamma}{\gamma + 1}}(\Omega; \R^d))
\ \mbox{for some}\ \gamma > 1,
\]
see \cite[Theorem 3.8]{AbbFeiNov}.

\item {\bf Compatibility.} Any dissipative solution $[\vr, \vu]$, $\vr > 0$, that is continuously differentiable is in fact a classical solution of the problem, in particular
\[
\mathfrak{R} = 0, \ \mathbb{S}_{\rm eff} = \mathbb{S}, \ \mathbb{S} \in \partial F(\Ds \vu),
\]
see \cite[Theorem 4.1]{AbbFeiNov}.

\item {\bf Weak--strong uniqueness.}
A dissipative solution coincides with the strong solution emanating from the same initial data as long as the latter solution
exists, see \cite[Theorem 6.3]{AbbFeiNov}.

\end{itemize}

The terminology ``dissipative solution'' was first used by Lions \cite{LI} in the context of the Euler system, where the equations
are simply replaced by the associated relative energy inequality. Brenier \cite{Breni} proposed an alternative approach
to construct generalized solutions of the Euler system via maximization of a concave functional. Our concept of dissipative solution
is closer to the measure--valued solution in the spirit of DiPerna's pioneering work \cite{DiP2}, see also the monograph
M\' alek et al. \cite{MNRR} and the references cited therein. The key observation is that the oscillation and concentration defects can be
conveniently unified giving rise to a single positively definite Reynolds stress, the trace of which is controlled by the energy
dissipation defect.

Anticipating that dissipative solutions are possibly not uniquely determined by the initial data $[\vr_0, \vm_0]$, we identify a
smaller class of \emph{maximal} dissipative solutions -- the dissipative solutions with a maximal rate of energy dissipation.
We show that maximal dissipative solutions exist for any finite energy initial data, and, in addition, they enjoy the following remarkable property:
\begin{equation} \label{P12}
\| \mathfrak{R}(\tau) \|_{\mathcal{M}(\Ov{\Omega}; \R^{d \times d}_{\rm sym})} \to 0
\ \mbox{as}\ \tau \to \infty.
\end{equation}
In other words, the maximal dissipative solutions behave like the conventional weak solutions in the long run.

Finally, imposing some technical hypotheses on $F$ and $p$ we show that any maximal dissipative solution $[\vr, \vm]$ approaches
an equilibrium state for large times:
\begin{equation} \label{P13}
\vm = \vr \vu (\tau, \cdot) \to 0 \ \mbox{in}\ L^{\frac{2 \gamma}{\gamma + 1}}(\Omega; \R^d),\
\vr(\tau, \cdot) \to \Ov{\vr} \ \mbox{in}\ L^\gamma(\Omega) \ \mbox{as}\ \tau \to \infty.
\end{equation}

The paper is organized as follows. In Section \ref{W}, we recall the concept of dissipative solution and introduce the class of
solutions with maximal energy dissipation. In Section \ref{L}, we study the long--time behavior of maximal solutions.
In particular, we show \eqref{P12}, see Theorem \ref{TL1}. In Sections \ref{C}, we introduce additional hypotheses to be imposed
on $F$ and $p$ as well as on the dissipative solution in order to prove \eqref{P13}. Then we show a general result on convergence
for a special class of dissipative solution, see Theorem \ref{TC1}.
Finally, in Section \ref{CC} we show unconditional convergence to equilibrium for the dissipative solutions imposing only
extra restrictions on $p$ and $F$, see Theorem
\ref{TCC}. The paper is concluded by a short discussion concerning
possible extensions to driven systems in Section \ref{D}.

\section{Dissipative solutions}
\label{W}

We start by recalling the basic restrictions on the structural properties of $F$ and $p$ introduced in \cite{AbbFeiNov}.

The pressure $p = p(\vr)$, with the associated pressure potential $P(\vr)$,
\[
P'(\vr) \vr - P(\vr) = p(\vr),
\]
satisfy
\begin{equation} \label{S1}
\begin{split}
&p \in C[0, \infty) \cap C^2(0, \infty),\
p(0) = 0, \ p'(\vr) > 0 \ \mbox{for}\ \vr > 0,\ P(0) = 0, \\
&P - \underline{a} p,\ \Ov{a} p - P\ \mbox{are convex functions
for certain constants}\ \underline{a} > 0, \ \Ov{a} > 0.
\end{split}
\end{equation}
Note that the standard isentropic pressure $p(\vr) = a \vr^\gamma$ satisfies \eqref{S1} with
\[
\underline{a} = \Ov{a} = \frac{1}{\gamma - 1}.
\]
Without loss of generality, we may fix
\[
\underline{a} = \sup \left\{ a > 0 \ \Big|\ P - {a} p \ \mbox{is convex} \right\},\
\Ov{a} = \inf \left\{ a > 0 \ \Big|\ a p - P \ \mbox{is convex} \right\}.
\]
As shown in \cite[Section 2.1.1]{AbbFeiNov}, we have
\begin{equation} \label{S2}
P(\vr) \geq a \vr^\gamma \ \mbox{for certain}\ a > 0,\ \gamma = 1 + \frac{1}{\Ov{a}}, \ \mbox{and all}\ \vr \geq 1.
\end{equation}

The dissipative potential satisfies
\begin{equation} \label{S3}
F: \R^{d \times d}_{\rm sym} \to [0, \infty) \ \mbox{is a (proper) convex function},\ F(0) = 0.
\end{equation}
Moreover, for any $R > 0$ there exists a (Young) function $A_R$ satisfying
\begin{itemize}
\item $A: [0, \infty) \to [0, \infty)$ convex,
\item $A$ increasing,
\item $A(0) = 0$,
\item $a_1 A(z) \leq A(2z) \leq a_2 A(z) \ \mbox{for any}\ z \in [0, \infty), \ \mbox{where}\ a_1 > 2,\ a_2 < \infty$,
\end{itemize}
such that
\begin{equation} \label{S5}
F(\mathbb{D} + \mathbb{Q}) - F(\mathbb{D}) - \mathbb{S}: \mathbb{Q} \geq
A_R \left( \left| \mathbb{Q} - \frac{1}{d} {\rm tr}[\mathbb{Q}] \mathbb{I} \right| \right)
\end{equation}
for all $\mathbb{D}, \mathbb{S}, \mathbb{Q} \in R^{d \times d}_{\rm sym}$ such that
\[
|\mathbb{D}| \leq R,\ \mathbb{S} \in \partial {F}(\mathbb{D}).
\]
As shown in \cite[Section 2.1.2]{AbbFeiNov},
it follows from \eqref{S5} that there exist $\mu > 0$ and $q > 1$ such that
\begin{equation} \label{S6}
F(\mathbb{D}) \geq \mu \left| \mathbb{D} - \frac{1}{d} {\rm tr}[\mathbb{D}] \mathbb{I} \right|^q
\ \mbox{for all}\ |\mathbb{D}| > 1.
\end{equation}

We are ready to introduce the concept of \emph{dissipative solution} to the problem \eqref{P1}--\eqref{P5}.

\begin{Definition}[Dissipative solution] \label{SD1}
Let $\Omega \subset \R^d$, $d=2,3$ be a bounded Lipschitz domain. The quantity $[\vr, \vu]$ is called
\emph{dissipative solution} of the problem \eqref{P1}--\eqref{P5} in $[0, \infty) \times \Omega$ if:
\begin{itemize}
\item
\[
\begin{split}
\vr \geq 0, \vr &\in C_{\rm weak, loc}([0, \infty); L^\gamma(\Omega)),\\
\vu &\in L^q([0, \infty); W^{1,q}_0(\Omega; \R^d)),\
\vm \equiv \vr \vu \in C_{\rm weak,loc}([0, \infty); L^{\frac{2 \gamma}{\gamma + 1}}(\Omega; \R^d));
\end{split}
\]
\item
the integral identity
\begin{equation} \label{W1}
\left[ \intO{ \vr \varphi } \right]_{t = 0}^{t = \tau} =
\int_0^\tau \intO{ \Big[ \vr \partial_t \varphi + \vr \vu \cdot \Grad \varphi \Big] } \dt,\
\vr(0, \cdot) = \vr_0,
\end{equation}
holds
for any $\tau \geq 0$, and any test function $\varphi \in C^1_{\rm loc}([0,\infty) \times \Ov{\Omega})$;
\item
there exist
\[
\mathbb{S} \in L^1_{\rm loc}([0,T) \times \Omega; \R^{d \times d}_{\rm sym}),\ \mathfrak{R} \in L^\infty(0,T; \mathcal{M}^+(\Ov{\Omega}; \R^{d \times d}_{\rm sym})),
\]
such that the integral identity
\begin{equation} \label{S8}
\begin{split}
\left[ \intO{ \vr \vu \cdot \bfphi } \right]_{t=0}^{t = \tau} &=
\int_0^\tau \intO{ \Big[ \vr \vu \cdot \partial_t \bfphi + \vr \vu \otimes \vu : \Grad \bfphi
+ p(\vr) \Div \bfphi - \mathbb{S} : \Grad \bfphi \Big] }\\
&+ \int_0^\tau \int_{{\Omega}} \Grad \bfphi : \D \ \mathfrak{R}(t) \ \dt,\ \vr \vu(0,\cdot) = \vm_0,
\end{split}
\end{equation}
holds for any $\tau \geq 0$ and any test function $\bfphi \in C^1_c([0,\infty) \times {\Omega}; \R^d)$;
\item
the energy inequality
\begin{equation} \label{S7}
\begin{split}
\intO{ E\left( \vr, \vm \Big| \Ov{\vr} \right) (\tau, \cdot) } &+ D \int_{\Ov{\Omega}} \D\ {\rm tr} [\mathfrak{R}](\tau) +
\int_0^\tau \intO{ \Big( F(\Ds \vu) + F^* (\mathbb{S}) \Big) } \dt \\ &\leq \intO{ E\left( \vr_0, \vm_0 \Big| \Ov{\vr} \right) }
\end{split}
\end{equation}
holds for a.e. $\tau \geq 0$, where
\[
D = \min \left\{ \frac{1}{2}; \frac{\underline{a}}{d} \right\}.
\]
\end{itemize}
\end{Definition}

The dissipative solutions have been introduced in \cite{AbbFeiNov}, specifically Definition 2.1 and Remarks 2.2, 2.3. The constant $D$
was computed explicitly as pointed out in \cite[Remark 2.3]{AbbFeiNov}. In \eqref{S7}, the kinetic energy
is defined as a convex l.s.c. function of $(\vr, \vm) \in \R^{d+1}$,
\[
\frac{1}{2} \frac{|\vm|^2}{\vr} = \left\{ \begin{array}{l} \frac{1}{2} \frac{|\vm|^2}{\vr}
\ \mbox{if}\ \vr > 0, \\ \\ 0 \ \mbox{if} \ \vr = 0, \vm = 0, \\ \\
\infty \ \mbox{otherwise.} \end{array} \right.
\]

\subsection{Turbulent energy and maximal dissipation}

In order to define the maximal solutions, we first introduce the \emph{turbulent energy} $\mathcal{E}$,
\begin{itemize}
\item
\[
\mathcal{E} \in L^\infty(0, \infty);
\]
\item
\[
\begin{split}
\intO{ \left[ \frac{1}{2} \frac{|\vm|^2}{\vr} + P(\vr) - P'(\Ov{\vr})(\vr - \Ov{\vr}) - P(\Ov{\vr}) \right]
} &+ D \int_{\Ov{\Omega}} \D\ {\rm tr} [\mathfrak{R}] \leq \mathcal{E} \\
\leq \int_\Omega \Big[ \frac{1}{2} \frac{|\vm_0|^2}{\vr_0} + P(\vr_0) - P'(\Ov{\vr})&(\vr_0 - \Ov{\vr}) - P(\Ov{\vr}) \Big]
\ \dx
\ \mbox{a.e. in}\ (0, \infty);
\end{split}
\]
\item
\begin{equation} \label{S10}
\frac{{\rm d}}{{\rm d}t} \mathcal{E} \leq - \intO{ \left[ F(\Ds \vu) + F^*(\mathbb{S}) \right] }
\ \mbox{in}\ \mathcal{D}'(0, \infty).
\end{equation}
\end{itemize}

In general, the turbulent energy $\mathcal{E}$ is not uniquely determined by $[\vr, \vu]$ and Reynolds defect $\mathfrak{R}$, however, at least
one turbulent energy exists. Indeed,
in view of the energy inequality \eqref{S7}, we may take
\[
\mathcal{E}(\tau) = \intO{ E\left( \vr_0, \vm_0 \Big| \Ov{\vr} \right) } -
\int_0^\tau \intO{ \Big( F(\Ds \vu) + F^* (\mathbb{S}) \Big) } \dt,
\]
where the right--hand side is non--increasing. Moreover, given $[\vr, \vu]$ we can modify the Reynolds defect $\mathfrak{R}$,
\[
\mathfrak{R} \approx \mathfrak{R} + \chi(t) \mathbb{I},\ \chi \in L^\infty (0, \infty),\ \chi \geq 0,
\]
without changing the momentum balance \eqref{S8} in such a way that
\begin{equation} \label{S9}
\begin{split}
\intO{ \left[ \frac{1}{2} \frac{|\vm|^2}{\vr} + P(\vr) - P'(\Ov{\vr})(\vr - \Ov{\vr}) - P(\Ov{\vr}) \right](\tau, \cdot)
} + D \int_{\Ov{\Omega}} \D\ {\rm tr} [\mathfrak{R}](\tau) = \mathcal{E}(\tau)
\ \mbox{for a.a.}\ \tau \in (0, \infty) \\
\mathcal{E}(0+) = \intO{ E \left(\vr_0, \vm_0 \Big| \Ov{\vr} \right) }.
\end{split}
\end{equation}
In the rest of the paper, we restrict ourselves to the dissipative solutions for which the turbulent energy
$\mathcal{E}$ given by \eqref{S9} satisfies
\eqref{S10}. For definitness, we identify $\mathcal{E}$ with its c\` adl\` ag version,
\[
\mathcal{E}(\tau) = \mathcal{E}(\tau+).
\]

Motivated by Dafermos \cite{Daf55}, \cite{Daf4}, we introduce the concept of maximal solution.
Given two dissipative solutions $[\vr_1, \vu_1]$, $[\vr_2, \vu_2]$ emanating from the same initial data
$[\vr_0, \vm_0]$, with the associated turbulent energy $\mathcal{E}_1$, $\mathcal{E}_2$, we say that
\begin{equation} \label{S11}
[\vr^1, \vu^1] \prec [\vr^2, \vu^2] \ \Leftrightarrow \ \mathcal{E}^1  \leq \mathcal{E}^2 \ \mbox{in}\ [0, \infty).
\end{equation}
To define a maximal solution we first introduce the set
\[
\begin{split}
\mathcal{U}[\vr_0, \vm_0] = \Big\{ &[\vr, \vu, \mathcal{E}] \ \Big|\
[\vr, \vu] \ \mbox{--a dissipative solutions with the initial data}\
[\vr_0, \vm_0] \\ &\mbox{and the associated turbulent energy}\ \mathcal{E} \Big\}
\end{split}
\]

\begin{Definition}[Maximal solution] \label{SD2}

We say that a dissipative solution $[\vr, \vu]$ emanating from the initial data $[\vr_0, \vm_0]$ with the associated
turbulent energy $\mathcal{E}$ is \emph{maximal}
if it is minimal with respect to the relation ``$\prec$'' among all dissipative solutions in $\mathcal{U}[\vr_0, \vm_0]$.
More specifically, if $[\tvr, \tvu]$ is another dissipative solution starting from $[\vr_0, \vm_0]$ with the
associated turbulent energy $\widetilde{\mathcal{E}}$ satisfying
\[
\widetilde{\mathcal{E}} \leq \mathcal{E} \ \mbox{then}\ \widetilde{\mathcal{E}} = \mathcal{E}.
\]

\end{Definition}

\begin{Remark} \label{RD1}

Seeing that
\[
\intO{ P'(\Ov{\vr})(\vr - \Ov{\vr}) - P(\Ov{\vr}) } = - \intO{P(\Ov{\vr})} \ \mbox{- a constant}
\]
we may consider the turbulent energy in a more concise form
\[
\mathcal{E}(\tau) = \intO{ \left[ \frac{1}{2} \frac{|\vm|^2}{\vr} + P(\vr) \right](\tau, \cdot)
} + D \int_{\Ov{\Omega}} \D\ {\rm tr} [\mathfrak{R}](\tau)
\]
independent of the total mass $M = \Ov{\vr} |\Omega|$.

\end{Remark}

\subsection{Existence of maximal solutions}
\label{EM}

The existence of a maximal solution can be proved following the line of arguments used in \cite{BreFeiHof19}. To begin, it is easy to observe that
a minimizer of the functional
\[
I [ \vr, \vu, \mathfrak{R} ] = \int_0^\infty \exp(-t) \mathcal{E} (t) \ \dt
\]
over the set of all dissipative solutions in $\mathcal{U}[\vr_0, \vm_0]$ is a maximal solution in the sense of
Definition \ref{SD2}. Here, the turbulent energy $\mathcal{E}$ is given in terms of $[\vr, \vu, \mathfrak{R}]$ through
\eqref{S9}.

Let $\left\{ [\vr_n, \vu_n] \right\}_{n=1}^\infty$, with the associated $\{ \mathcal{E}_n \}_{n=1}^\infty$, be a minimizing sequence of $I$ on
$\mathcal{U}{[\vr_0, \vm_0]}$. In view of the uniform bounds resulting from the energy inequality \eqref{S7} and
Helly's theorem, we may extract a suitable subsequence (not relabeled) such that
\[
\begin{split}
\vr_n &\to \vr \ \mbox{in} \ C_{\rm weak, loc}([0, \infty); L^\gamma (\Omega)),\\
\vu_n &\to \vu \ \mbox{weakly in}\ L^q([0,\infty); W^{1,q}_0 (\Omega; \R^d)), \\
\vr_n \vu_n \equiv \vm_n &\to \vm \ \mbox{in} \ C_{\rm weak, loc}([0, \infty); L^{\frac{2 \gamma}{\gamma + 1}} (\Omega)),\\
\mathcal{E}_n &\to \mathcal{E} \ \mbox{pointwise in}\ [0, \infty),\\
\mathbb{S}_n &\to \mathbb{S} \ \mbox{weakly in}\ L^1_{\rm loc}([0,\infty) \times \Omega; \R^{d \times d}_{\rm sym}),\\
\mathfrak{R}_n &\to \mathfrak{R}^\infty \ \mbox{weakly-(*) in}\ L^\infty(0,\infty; \mathcal{M}(\Ov{\Omega}; \R^{d \times d}_{\rm sym})),\\
\mathfrak{R}^{\rm conv}_n \equiv \left( 1_{\vr_n > 0} \frac{\vm_m \otimes \vm_n}{\vr_n} -
1_{\vr > 0} \frac{ \vm \otimes \vm }{\vr} \right) &\to \mathfrak{R}^{\rm conv} \ \mbox{weakly-(*) in}\ L^\infty(0,\infty;
\mathcal{M}(\Ov{\Omega}; \R^{d \times d}_{\rm sym})),\\
\mathfrak{R}^p_n \equiv ( p(\vr_n) - p(\vr) ) &\to \mathfrak{R}^p\ \mbox{weakly-(*) in}\ L^\infty(0,\infty;
\mathcal{M}(\Ov{\Omega})),\\
\mathfrak{R}^{\rm kin}_n \equiv \left( \frac{1}{2} \frac{|\vm_n|^2}{\vr_n} - \frac{1}{2} \frac{|\vm|^2}{\vr} \right)
&\to \mathfrak{R}^{\rm kin}\ \mbox{weakly-(*) in}\ L^\infty(0,\infty;
\mathcal{M}(\Ov{\Omega})),\\
\mathfrak{R}^P_n \equiv ( P(\vr_n) - P(\vr) ) &\to \mathfrak{R}^P\ \mbox{weakly-(*) in}\ L^\infty(0,\infty;
\mathcal{M}(\Ov{\Omega})).
\end{split}
\]

Repeating the arguments used in \cite[Section 3.4]{AbbFeiNov}, we successively deduce that
\[
\vm = \vr \vu;
\]
\[
\mathcal{E} = \intO{ \left[ \frac{1}{2} \frac{|\vm|^2}{\vr} + P(\vr) - P'(\Ov{\vr})(\vr - \Ov{\vr}) - P(\Ov{\vr}) \right] }
+ \int_{\Ov{\Omega}} \D \mathfrak{R}^{\rm kin} + \int_{\Ov{\Omega}} \D \mathfrak{R}^{\rm P} +
D \int_{\Ov{\Omega}} \D \ {\rm tr}[\mathfrak{R}^\infty];
\]
\[
\left[ \intO{ \vr \varphi } \right]_{t = 0}^{t = \tau} =
\int_0^\tau \intO{ \Big[ \vr \partial_t \varphi + \vr \vu \cdot \Grad \varphi \Big] } \dt,\
\vr(0, \cdot) = \vr_0,
\]
for any $\tau \geq 0$, and any test function $\varphi \in C^1_{\rm loc}([0,\infty) \times \Ov{\Omega})$;
\[
\begin{split}
\left[ \intO{ \vr \vu \cdot \bfphi } \right]_{t=0}^{t = \tau} &=
\int_0^\tau \intO{ \Big[ \vr \vu \cdot \partial_t \bfphi + \vr \vu \otimes \vu : \Grad \bfphi
+ p(\vr) \Div \bfphi - \mathbb{S} : \Grad \bfphi \Big] }\\
+ \int_0^\tau \int_{\Ov{\Omega}}
\Grad \bfphi : &\D \Big[ \mathfrak{R}^{\rm conv} + \mathfrak{R}^p {\rm Id} \Big](t) \dt +
\int_0^\tau \int_{{\Omega}} \Grad \bfphi : \D \ \mathfrak{R}^\infty(t) \ \dt,\ \vr \vu(0,\cdot) = \vm_0,
\end{split}
\]
for any $\tau \geq 0$ and any test function $\bfphi \in C^1_c([0,\infty) \times {\Omega}; \R^d)$;
\[
\begin{split}
\frac{{\rm d}}{{\rm d}t} \mathcal{E} &\leq  -
\intO{ \Big( F(\Ds \vu) + F^* (\mathbb{S}) \Big) } \dt \ \mbox{in}\ \mathcal{D}'(0, \infty),\\
\mathcal{E}(0+) &= \intO{ E\left( \vr_0, \vm_0 \Big| \Ov{\vr} \right) }.
\end{split}
\]

Next, the convexity hypothesis \eqref{S1} implies that
\[
\mathfrak{R}^P \geq \frac{\underline{a}}{d} {\rm tr}[\mathfrak{R}^p {\rm Id}],
\ \mbox{while, obviously,}\ \mathfrak{R}^{\rm kin} \geq \frac{1}{2} {\rm tr}[\mathfrak{R}^{\rm conv}]
\]
Consequently, introducing a new Reynolds stress
\[
\mathfrak{R} = \mathfrak{R}^{\rm conv} + \mathfrak{R}^p \mathbb{I} + \mathfrak{R}^\infty
\in L^\infty(0, \infty; \mathcal{M}^+(\Ov{\Omega}; \R^{d \times d}_{\rm sym}))
\]
we may infer that
\[
\mathcal{E}(\tau) \geq \intO{ \left[ \frac{1}{2} \frac{|\vm|^2}{\vr} + P(\vr) - P'(\Ov{\vr})(\vr - \Ov{\vr}) - P(\Ov{\vr}) \right] }
+ D \int_{\Ov{\Omega}} \D \ {\rm tr}[\mathfrak{R}]
\]
for a.e. $\tau \in (0, \infty)$. Thus modifying $\mathfrak{R}$
\[
\mathfrak{R} \approx \mathfrak{R} + \chi \mathbb{I},\ \chi \in L^\infty(0,\infty),\ \chi \geq 0,
\]
we achieve
\[
\mathcal{E}(\tau) = \intO{ \left[ \frac{1}{2} \frac{|\vm|^2}{\vr} + P(\vr) - P'(\Ov{\vr})(\vr - \Ov{\vr}) - P(\Ov{\vr}) \right] }
+ D \int_{\Ov{\Omega}} \D \ {\rm tr}[\mathfrak{R}]
\]
for a.e. $\tau \in (0, \infty)$. In other words, $[\vr, \vu]$ is a dissipative solution, with the associated turbulent energy
$\mathcal{E}$ minimizing the functional $I$; whence maximal.

We have shown the following result.

\begin{Proposition}[Existence of maximal solutions] \label{PS1}
Let $\Omega \subset \R^d$, $d=2,3$ be a bounded Lipschitz domain. Suppose that $F$ and $p$ comply with the hypotheses \eqref{S1},
\eqref{S3}, \eqref{S5}. Let the initial data $[\vr_0, \vm_0]$ be given,
\[
\vr_0 \geq 0,\ \intO{ E\left( \vr_0, \vm_0 \ \Big| \Ov{\vr} \right) } < \infty.
\]

Then the problem \eqref{P1}--\eqref{P5} admits a maximal dissipative solution $[\vr, \vu]$ in $(0, \infty) \times
\Omega$ in the sense specified in Definition \ref{SD2}, meaning minimal with respect to the relation $\prec$ in
$\mathcal{U}[\vr_0, \vm_0]$.

\end{Proposition}

\section{Long--time behavior of maximal solutions}
\label{L}

We are ready to state our main result concerning the long time behavior of maximal solutions.

\begin{Theorem}[Long time behavior] \label{TL1}
Let $\Omega \subset \R^d$, $d=2,3$ be a bounded Lipschitz domain. Suppose that $F$ and $p$ comply with the hypotheses \eqref{S1},
\eqref{S3}, \eqref{S5}. Let $[\vr, \vu]$ be a maximal dissipative solution the problem \eqref{P1}--\eqref{P5}
in $(0, \infty) \times \Omega$ in $\mathcal{U}[\vr_0, \vm_0]$, with the associate Reynolds defect $\mathfrak{R}$ and the turbulent energy $\mathcal{E}$,
\[
\mathcal{E}(\tau) \to \mathcal{E}_\infty \ \mbox{as}\ \tau \to \infty.
\]

Then
\[
\intO{ \left[ \frac{1}{2} \frac{|\vm|^2}{\vr} + P(\vr) - P'(\Ov{\vr})(\vr - \Ov{\vr}) - P(\Ov{\vr}) \right](\tau, \cdot)
} \to \mathcal{E}_\infty \ \mbox{as}\ \tau \to \infty,
\]
in particular,
\[
{\rm ess}\lim_{\tau \to \infty} \left\| \mathfrak{R}(\tau) \right\|_{\mathcal{M}(\Ov{\Omega}; \R^{d \times d}_{\rm sym})}
= 0.
\]

\end{Theorem}

\begin{proof}

As the energy $E \left(\vr, \vm \ \Big| \Ov{\vr} \right)$ is convex and the functions $\vr$, $\vm$ weakly continuous in the
time variable, we have
\[
\mathcal{E}_\infty \geq \limsup_{\tau \to \infty} \intO{ \left[ \frac{1}{2} \frac{|\vm|^2}{\vr} + P(\vr) - P'(\Ov{\vr})(\vr - \Ov{\vr}) - P(\Ov{\vr}) \right](\tau, \cdot)
}.
\]
Consequently, it is enough to show
\begin{equation} \label{L1}
\intO{ \left[ \frac{1}{2} \frac{|\vm|^2}{\vr} + P(\vr) - P'(\Ov{\vr})(\vr - \Ov{\vr}) - P(\Ov{\vr}) \right](T, \cdot)}
\equiv \intO{ E \left( \vr, \vm \Big| \Ov{\vr} \right)(T, \cdot) }
\geq \mathcal{E}_\infty
\end{equation}
for any $T \geq 0$.

Arguing by contradiction we suppose there exists $T \geq 0$ such that
\[
\mathcal{E}_\infty > \intO{ E \left( \vr, \vm \Big| \Ov{\vr} \right)(T, \cdot) }.
\]
In accordance with Proposition \ref{PS1}, the problem \eqref{P1}--\eqref{P5} admits a dissipative solution
$[\vr_T, \vu_T]$ in $(T, \infty) \times \Omega$, starting from the initial data $[\vr(T, \cdot), \vm(T, \cdot)]$, and such that the associated turbulent energy $\mathcal{E}_T$ satisfies
\begin{equation} \label{L2}
\mathcal{E}_T (\tau) \leq \intO{ E \left( \vr, \vm \Big| \Ov{\vr} \right)(T, \cdot) } < \mathcal{E}_\infty \leq
\mathcal{E} (\tau) \ \mbox{for all}\ \tau \geq T.
\end{equation}

Finally, we define a new dissipative solution solution $[\tvr, \tvu]$,
\[
[\tvr, \tvu](\tau, \cdot) = \left\{ \begin{array}{l} {[\vr, \vu]}(\tau, \cdot) \ \mbox{if}\ 0 \leq \tau < T,\\ \\
{[\vr_T, \vu_T]}(\tau, \cdot) \ \mbox{if}\ \tau \geq T. \end{array} \right.
\]
with the associated turbulent energy
\[
\widetilde{\mathcal{E}}(\tau) = \left\{ \begin{array}{l} \mathcal{E} (\tau) \ \mbox{if}\ 0 \leq \tau < T,\\ \\
\mathcal{E}_T (\tau) \ \mbox{if}\ \tau \geq T. \end{array} \right.
\]
In view of \eqref{L2}, however, $[\tvr, \tvu] \prec [\vr, \vu]$ and $[\vr, \vu]$ is not maximal in contrast
with our hypothesis.

\end{proof}

\section{Concergence to equilibria}
\label{C}

In order to establish convergence of dissipative solutions to equilibria, extra hypotheses must be imposed on the structural properties
of $F$ and $p$, specifically, on the the exponents $\gamma$ and $q$ appearing in \eqref{S2} and \eqref{S6}, respectively.

\subsection{Renormalization}

In addition to \eqref{W1}, we need its renormalized version
\begin{equation} \label{A1}
\left[ \intO{ B(\vr) \varphi } \right]_{t = 0}^{t = \tau} =
\int_0^\tau \intO{ \left[ B(\vr) \partial_t \varphi + B(\vr) \vu \cdot \Grad \varphi
+ \Big( B(\vr) - B'(\vr) \vr \Big) \Div \vu \right] } \dt
\end{equation}
to be satisfied
for any $0 \leq \tau \leq T$, any test function $\varphi \in C^1_c([0,\infty) \times \Ov{\Omega})$, and any $B \in C^1(\R)$,\
$B' \in C_c(\R)$.

In accordance with the DiPerna--Lions theory \cite{DL}, the renormalized equation \eqref{A1} follows from \eqref{W1} as soon as
\begin{equation} \label{A2}
\frac{1}{\gamma} + \frac{1}{q} \leq 1.
\end{equation}

\subsection{Bounds on kinetic energy}
\label{BK}

In order to prove convergence to equilibria, better bounds on the kinetic energy are necessary. More specifically, we need
\begin{equation} \label{A3}
\sup_{T \geq 0} \int_T^{T+1} \intO{ \vr^\alpha |\vu|^{2\alpha} } \dt < \infty \ \mbox{for some}\ \alpha > 1.
\end{equation}
To obtain \eqref{A3}, we write
\[
\| \vr |\vu|^2 \|_{L^\alpha (\Omega)} \leq  \| \vr \vu  \|_{L^{\frac{2 \gamma}{\gamma + 1}}(\Omega; \R^d)}
\| \vu \|_{L^p(\Omega; \R^d)},\ \frac{1}{\alpha} = \frac{\gamma + 1}{2 \gamma} + \frac{1}{p}.
\]
On the other hand, by Sobolev's embedding theorem,
\[
\begin{split}
\| \vu \|_{L^\infty(\Omega; \R^d)} &\aleq \| \vu \|_{W^{1,q}(\Omega; \R^d)} \ \mbox{if}\ q > d,\\
\| \vu \|_{L^p(\Omega; \R^d)} &\aleq \| \vu \|_{W^{1,q}(\Omega; \R^d)} \ \mbox{for any}
\ 1 \leq p < \infty \ \mbox{if} \ q = d,\\
\| \vu \|_{L^p(\Omega; \R^d)} &\aleq \| \vu \|_{W^{1,q}(\Omega; \R^d)}\ \mbox{for}\
1 \leq p \leq \frac {dq}{d - q} \ \mbox{if}\ q < d.
\end{split}
\]
Consequently, the desired bound \eqref{A3} follows from the energy inequality \eqref{S7} as soon as
\begin{equation} \label{A4}
\frac{\gamma + 1}{2 \gamma} + \frac{d - q}{dq} < 1.
\end{equation}

Moreover, a short inspection of the existence proof in \cite{AbbFeiNov} reveals that the traceless part of the Reynolds stress
$\mathfrak{R}$,
\[
\mathfrak{R} - \frac{1}{d}{\rm tr} [\mathfrak{R}] \ \mathbb{I} = \Ov{ \frac{ \vm \otimes \vm }{\vr} -
\frac{1}{d} \frac{|\vm|^2}{\vr} \mathbb{I} } - \left( \frac{\vm \otimes \vm}{\vr} - \frac{1}{d} \frac{|\vm|^2}{\vr} \mathbb{I} \right),
\]
admits a bound similar to \eqref{A3}, namely,
\begin{equation} \label{A5}
\sup_{T \geq 0} \int_T^{T+1} \intO{ \left| \mathfrak{R} - \frac{1}{d}{\rm tr} [\mathfrak{R}]
\mathbb{Id} \right|^\alpha } \dt < \infty \ \mbox{for some}\ \alpha > 1.
\end{equation}
In other words, under the hypothesis \eqref{A4}, there exists a dissipative solution with the associated Reynolds stress
$\mathfrak{R}$ satisfying \eqref{A5}.

\begin{Remark} \label{RC2}

As a matter of fact, the traceless part
\[
\mathfrak{R} - \frac{1}{d}{\rm tr} [\mathfrak{R}] \ \mathbb{I} = \Ov{ \frac{ \vm \otimes \vm }{\vr} -
\frac{1}{d} \frac{|\vm|^2}{\vr} \mathbb{I} } - \left( \frac{\vm \otimes \vm}{\vr} - \frac{1}{d} \frac{|\vm|^2}{\vr} \mathbb{I} \right)
\]
actually \emph{vanishes} for the dissipative solutions constructed by the method of \cite{AbbFeiNov}. Indeed the sequence of approximate solutions $\left\{ [\vr_n, \vu_n] \right\}_{n=1}^\infty$ constructed in \cite[Section 3.4]{AbbFeiNov} satisfies
\[
\vr_n \vu_n \to \vr \vu \ \mbox{in}\ C_{\rm weak}([0,T]; L^{\frac{2 \gamma}{\gamma+ 1}}(\Omega; \R^d)),\
\vu_n \to \vu \ \mbox{weakly in}\ L^q([0,T]; W^{1,q}_0(\Omega; \R^d))
\]
for arbitrary $T > 0$. Moreover, the hypothesis \eqref{A4} implies that
\[
L^{\frac{2 \gamma}{\gamma + 1}}(\Omega; {\rm weak}) \hookrightarrow\hookrightarrow W^{-1,q}(\Omega);
\]
whence
\[
\Ov{ \frac{ \vm \otimes \vm }{\vr} -
\frac{1}{d} \frac{|\vm|^2}{\vr} \mathbb{I} } - \left( \frac{\vm \otimes \vm}{\vr} - \frac{1}{d} \frac{|\vm|^2}{\vr} \mathbb{I} \right)
= 0
\]

\end{Remark}

\subsection{Bounds on the viscous stress}

In addition to the lower bound \eqref{S5}, we suppose that
\begin{equation} \label{A6}
F (\mathbb{D}) \aleq 1 + |\mathbb{D}|^r\ \mbox{for some}\ r < \infty, \ \partial F(0) = \{ 0 \},
\end{equation}
which implies
\begin{equation} \label{A7}
F^*(\mathbb{S}) > 0 \ \mbox{for all}\ \mathbb{S} \ne 0,\
F^*(\mathbb{S}) \ageq |\mathbb{S}|^\alpha \ \mbox{for some}\ \alpha > 1, \ \mbox{and all}\
|\mathbb{S}| \geq 1.
\end{equation}

\subsection{Convergence}

We are ready to state our main result concerning convergence to equilibria of dissipative solutions.

\begin{Theorem}[Convergence to equilibria] \label{TC1}
Let $\Omega \subset \R^d$, $d=2,3$ be a bounded Lipschitz domain. Suppose that $F$ and $p$ comply with the hypotheses \eqref{S1},
\eqref{S3}, \eqref{S5}, and \eqref{A6}. In addition, suppose that the exponents $\gamma$ and $q$ appearing
in \eqref{S2} and \eqref{S6}, respectively, satisfy
\begin{equation} \label{A8}
\frac{\gamma + 1}{2 \gamma} + \frac{d - q}{dq} < 1.
\end{equation}
Let $[\vr, \vu]$ be a dissipative solution to the problem \eqref{P1}--\eqref{P5} satisfying the renormalized
equation of continuity \eqref{A1}, with the associated Reynolds stress $\mathfrak{R}$ such that
\begin{equation} \label{A9}
\sup_{T \geq 0} \int_T^{T+1} \intO{ \left| \mathfrak{R} - \frac{1}{d}{\rm tr} [\mathfrak{R}]
\mathbb{Id} \right|^\alpha } \dt < \infty \ \mbox{for some}\ \alpha > 1.
\end{equation}
Finally, suppose that
\begin{equation} \label{A9a}
{\rm ess}\lim_{\tau \to \infty} \left\| \mathfrak{R}(\tau) \right\|_{\mathcal{M}(\Ov{\Omega}; \R^{d \times d}_{\rm sym})}
= 0.
\end{equation}

Then
\[
\vr \vu (\tau, \cdot) \to 0 \ \mbox{in}\ L^{\frac{2 \gamma}{\gamma + 1}}(\Omega; \R^d),\
\vr(\tau, \cdot) \to \Ov{\vr}\ \mbox{in}\ L^\gamma(\Omega), \ \Ov{\vr} = \frac{1}{|\Omega|} \intO{ \vr_0},
\]
as $\tau \to \infty$.
\end{Theorem}

\begin{Remark} \label{RC1}

As pointed out in Section \ref{BK}, the problem \eqref{P1}--\eqref{P5} admits a dissipative solution satisfying
\eqref{A9} as soon as the hypothesis \eqref{A8} holds.

\end{Remark}

\begin{Remark} \label{RC3}

As stated in Theorem \ref{TL1}, the hypothesis \eqref{A9a} holds for any \emph{maximal} dissipative solution in the sense of
Definition \ref{SD2}.

\end{Remark}

The rest of this section is devoted to the proof of Theorem \ref{TC1}, carried over in several steps.

\subsubsection{Uniform pressure estimates}

In order to derive the estimates implying equi--integrability of the pressure $p(\vr)$, we introduce the so--called
Bogovskii operator $\mathcal{B}$ that may be seen as a suitable branch of the inverse divergence $\Div^{-1}$,
see Bogovskii \cite{BOG}. For reader's convenience, we recall the basic properties of the operator $\mathcal{B}$
proved in Bogovskii \cite{BOG}, Galdi \cite{GAL}, and Geissert, Heck, and Hieber \cite{GEHEHI} (see also \cite[Theorem 11.17] {FeNo6A}).

\begin{itemize}
\item
\[
\mathcal{B}: \left\{ f \in L^p(\Omega)\ \Big| \ \intO{ f } = 0 \right\} \to
W^{1,p}_0(\Omega; \R^d)
\]
is a bounded linear operator for any $1 < p < \infty$, specifically,
\begin{equation} \label{A10a}
\| \mathcal{B} [f] \|_{W^{1,p}_0(\Omega;\R^d)} \aleq \| f \|_{L^p(\Omega)},\ 1 < p < \infty.
\end{equation}
\item
\[
\Div \mathcal{B}[f]  = f;
\]
\item if $f \in L^p(\Omega)$, $\intO{ f } = 0$, and, in addition,
\[
f = \Div \vc{g},\ \vc{g} \in L^r(\Omega; \R^d), \ \Div \vc{g} \in L^p(\Omega),\ \vc{g} \cdot \vc{n}|_{\partial \Omega} = 0,
\]
then
\begin{equation} \label{A10b}
\| \mathcal{B}[f] \|_{L^r(\Omega; \R^d)} \aleq \| \vc{g} \|_{L^r(\Omega; \R^d)},\ 1 < r < \infty;
\end{equation}
\item
if $f \in W^{k,p}(\Omega)$, k=1,2,\dots, $1 < p < \infty$, $\intO{ f } = 0$, then
$\mathcal{B}[f] \in W^{k+1,p}(\Omega; \R^d)$.

\end{itemize}

It follows from \eqref{A1} that the renormalized equation of continuity holds in $(0, \infty) \times R^d$ provided $\vr$, $\vu$ are extended to be zero outside $\Omega$:
\[
\left[ \int_{\R^d} B(\vr) \varphi \ \dx \right]_{t = 0}^{t = \tau} =
\int_0^\tau \int_{\R^d} \left[ B(\vr) \partial_t \varphi + B(\vr) \vu \cdot \Grad \varphi
+ \Big( B(\vr) - B'(\vr) \vr \Big) \Div \vu \right] \ \dx \dt
\]
for any $0 \leq \tau \leq T$, any test function $\varphi \in C^1_c([0,\infty) \times \R^d)$, and any $B \in C^1(\R)$,\
$B' \in C_c(\R)$. Consequently, we may apply the standard regularization procedure via convolution with a family of
regularizing kernels $\{ \theta_\ep \}_{\ep > 0}$ in the $x-$variable to obtain
\[
\partial_t [B(\vr)]_\ep + \Div \left( [B(\vr)]_\ep \vu \right) +
\left[ \left( B'(\vr) \vr - B(\vr) \right) \Div \vu \right]_\ep =
E_\ep
\]
with the error term,
\[
E_\ep =
\Div \left( [B(\vr)]_\ep \vu \right) -
\left[ \Div (B(\vr) \vu) \right]_\ep,
\]
where we have denoted $[v]_\ep = \theta_\ep * v$. In view of \eqref{S6} and a version of Korn--Poincar\' e inequality,
\[
\vu \in L^q(0,T; W^{1,q}_0 (\Omega; \R^d)).
\]
As $B$ is bounded, we may use the DiPerna--Lions theory \cite{DL}, notably Friedrich's commutator lemma (see also
\cite[Lemma 11.12]{FeNo6A}), to conclude
\begin{equation} \label{A10}
E_\ep \to 0 \ \mbox{in}\ L^r_{\rm loc}(0,T; L^r(\Omega)) \ \mbox{as}\ \ep \to 0\ \mbox{for any}\ 1 \leq r < q.
\end{equation}

Next, we use
\[
\bfphi = \psi(t) \mathcal{B} \left[ [B(\vr)]_\ep - \frac{1}{|\Omega|} \intO{ [B(\vr)]_\ep } \right] ,\
\psi \in C^1_c(0, \infty),\ \psi \geq 0,
\]
as a test function in the momentum equation \eqref{S8}. After a straightforward manipulation, we obtain
\[
\begin{split}
\int_0^\infty \psi &\intO{ p(\vr) [B(\vr)]_\ep } \dt -
\frac{1}{|\Omega|} \int_0^\infty \psi \left( \intO{p(\vr)} \right) \left( \intO{ [B(\vr)]_\ep } \right) \dt\\
&+ \frac{1}{d} \int_0^\infty \psi \int_{\Omega}
[B(\vr)]_\ep \ \D {\rm tr}[\mathfrak{R}] \dt  - \frac{1}{d |\Omega|} \int_0^\infty \psi
\left( \int_\Omega \D {\rm tr}[\mathfrak{R}] \right) \left( \intO{ [B(\vr)]_\ep } \right) \dt \\ = -
\int_0^\infty &\psi \intO{ \Big[ \vr \vu \otimes \vu
- \mathbb{S}  + \left( \mathfrak{R}(t) - \frac{1}{d} {\rm tr}[\mathfrak{R}] \mathbb{I} \right)\Big]: \Grad \mathcal{B} \left[ [B(\vr)]_\ep - \frac{1}{|\Omega|} \intO{ [B(\vr)]_\ep } \right] } \dt \\
&- \int_0^\infty \partial_t \psi  \intO{ \vr \vu \cdot \mathcal{B} \left[ [B(\vr)]_\ep - \frac{1}{|\Omega|} \intO{ [B(\vr)]_\ep } \right]     } \dt\\
&+ \int_0^\infty \psi \intO{ \vr \vu \cdot \mathcal{B} [ \Div ( [B(\vr)]_\ep \vu) ] }\dt \\
&+ \int_0^\infty \psi \intO{ \vr \vu \cdot \mathcal{B} \left[ \left( B'(\vr) \vr - B(\vr) \right) \Div \vu -
\frac{1}{|\Omega|} \intO{ \left( B'(\vr) \vr - B(\vr) \right) \Div \vu } \right]_\ep }\dt\\
&- \int_0^\infty \psi \intO{ \vr \vu \cdot \mathcal{B}\left[ E_\ep - \frac{1}{|\Omega|} \intO{ E_\ep } \right] } \dt.
\end{split}
\]

Now observe that, in view of the hypothesis \eqref{A8}, the error estimate \eqref{A10}, and the regularization property of the operator
$\mathcal{B}$ stated in \eqref{A10a}, we may let $\ep \to 0$ obtaining
\begin{equation} \label{A12}
\begin{split}
\int_0^\infty \psi &\intO{ p(\vr) [B(\vr)] } \dt \leq
\frac{1}{|\Omega|} \int_0^\infty \psi \left( \intO{p(\vr)} \right) \left( \intO{ [B(\vr)] } \right) \dt\\
&+ \frac{1}{d |\Omega|} \int_0^\infty \psi
\left( \int_\Omega \D {\rm tr}[\mathfrak{R}] \right) \left( \intO{ [B(\vr)] } \right) \dt \\  -
\int_0^\infty &\psi \intO{ \Big[ \vr \vu \otimes \vu
- \mathbb{S}  + \left( \mathfrak{R}(t) - \frac{1}{d} {\rm tr}[\mathfrak{R}] \mathbb{I} \right)\Big]: \Grad \mathcal{B} \left[ [B(\vr)] - \frac{1}{|\Omega|} \intO{ [B(\vr)] } \right] } \dt \\
&- \int_0^\infty \partial_t \psi  \intO{ \vr \vu \cdot \mathcal{B} \left[ [B(\vr)] - \frac{1}{|\Omega|} \intO{ [B(\vr)] } \right]     } \dt\\
&+ \int_0^\infty \psi \intO{ \vr \vu \cdot \mathcal{B} [ \Div ( [B(\vr)] \vu) ] }\dt \\
&+ \int_0^\infty \psi \intO{ \vr \vu \cdot \mathcal{B} \left[ \left( B'(\vr) \vr - B(\vr) \right) \Div \vu -
\frac{1}{|\Omega|} \intO{ \left( B'(\vr) \vr - B(\vr) \right) \Div \vu } \right] }\dt.
\end{split}
\end{equation}

Finally, using the hypothesis \eqref{A9}, together with the bounds \eqref{A3}, \eqref{A7}, and the regularizing properties of
$\mathcal{B}$ stated in \eqref{A10a}, \eqref{A10b}, we may infer that
\begin{itemize}
\item validity of \eqref{A12} can be extended to $B(\vr) = \vr^\beta$ for some $\beta > 0$;
\item the integrals on the right--hand side of \eqref{A12} are uniformly bounded with respect to the time shifts of $\psi$.
\end{itemize}

We conclude that
\begin{equation} \label{A13}
\sup_{T \geq 0} \int_T^{T+1} \intO{ p(\vr) \vr^\beta } \dt < \infty \ \mbox{for some}\ \beta > 0,
\end{equation}
and, by virtue of the hypothesis \eqref{S1},
\begin{equation} \label{A14}
\sup_{T \geq 0} \int_T^{T+1} \intO{ P(\vr) \vr^\beta } \dt < \infty \ \mbox{for some}\ \beta > 0.
\end{equation}

\subsubsection{Convergence of density and momentum averages}

We are ready to show convergence to the equilibrium state. We introduce the time shifts
\[
\vr_n (t,x) = \vr(t + n,x),\ \vu_n (t,x) = \vu(t + n,x),\ \vm_n(t,x) = \vm(t + n,x) \ \mbox{etc.}
\]
In view of the bound
\[
\vu \in L^q(0, \infty; W^{1,q}_0(\Omega; \R^d)) < \infty,
\]
we have
\[
\vu_n \to 0 \ \mbox{in}\ L^q(0,1; W^{1,q}_0(\Omega; \R^d)).
\]
Moreover, as
\[
\vr \in L^\infty(0, \infty; L^\gamma(\Omega; \R^d)),
\]
and the kinetic energy is controlled by \eqref{A3}, we deduce that
\begin{equation} \label{A15}
\int_0^1 \intO{ \frac{1}{2} \frac{|\vm_n|^2}{\vr_n} } \to 0  \ \mbox{as}\ n \to \infty.
\end{equation}

The next step is to show strong a.e. convergence of $\{ \vr_n \}_{n=1}^\infty$. We start observing that
\[
\vr_n \to \vr_\infty \ \mbox{weakly-(*) in}\ L^\infty(0,1; L^\gamma(\Omega)), \ \vr_\infty \geq 0.
\]
Moreover, letting $n \to \infty$ in the equation of continuity \eqref{W1} we deduce
\begin{equation} \label{A15a}
\vr_\infty = \vr_\infty(x) \ \mbox{is independent of}\ t.
\end{equation}

Next, we perform the limit $n \to \infty$ in the momentum equation \eqref{S8}. Here, the crucial fact is that $\mathfrak{R}$ vanishes
as stated in \eqref{A9a}. Consequently, in accordance with \eqref{A7},
\[
\mathbb{S} \in L^\alpha (0, \infty; L^\alpha (\Omega; \R^{d \times d}_{\rm sym})),
\]
and we obtain
\begin{equation} \label{A16}
\Grad \Ov{p(\vr)} = 0 \ \mbox{in}\ \mathcal{D}'((0,1) \times \Omega),
\end{equation}
where
\[
p(\vr_n) \to \Ov{p(\vr)} \ \mbox{weakly in}\ L^1((0,T) \times \Omega).
\]
Here, similarly to \eqref{A15a}, the convergence holds up to a subsequence which we do not relabel.

In order to show strong convergence of $\{ \vr_n \}_{n=1}^\infty$, we consider \eqref{A12}, with
\[
B(\vr) = \vr^\alpha, \ \mbox{with} \ 0 < \alpha < \beta,
\]
where $\beta$ is the exponent in \eqref{A13}. Letting $n \to \infty$ in \eqref{A12} we obtain
\begin{equation} \label{A17}
\int_0^1 \intO{ \Ov{p(\vr) \vr^\alpha } } \dt \leq \frac{1}{|\Omega|} \intO{ \Ov{p(\vr)} } \intO{
\Ov{\vr^\alpha} },
\end{equation}
where, similarly to \eqref{A16}, the bar denotes the corresponding weak limits.

Now, testing \eqref{A16} on
\[
\psi \mathcal{B} \left[ \Ov{\vr^\alpha} - \frac{1}{|\Omega|} \intO{ \Ov{\vr^\alpha} }  \right]
\]
we obtain
\[
\int_0^1 \intO{ \Ov{p(\vr)} \ \Ov{\vr^\alpha} }  \dt = \frac{1}{|\Omega|} \intO{ \Ov{p(\vr)} } \intO{
\Ov{\vr^\alpha} },
\]
which, together with \eqref{A17}, gives rise to
\begin{equation} \label{A18}
\int_0^1 \intO{ \Ov{p(\vr) \vr^\alpha } } \leq \int_0^1 \intO{ \Ov{p(\vr)} \ \Ov{\vr^\alpha} }  \dt.
\end{equation}
As $p$ is strictly increasing, relation \eqref{A18} implies
\begin{equation} \label{A21}
\vr_n \to \vr \ \mbox{in measure in}\ (0,1) \times \Omega
\end{equation}
by means of the standard monotonicity argument (cf. \cite[Theorem 11.26]{FeNo6A}).

Finally, we deduce from \eqref{A16} that $\vr_\infty$ is independent of $x$; whence $\vr_\infty$ coincides with the constant equilibrium state
\[
\vr_\infty = \Ov{\vr}.
\]

Using \eqref{A15}, together with the uniform bound \eqref{A15} and the strong convergence of the density stated in
\eqref{A21}, we conclude that there is a sequence of times $\tau_n \to \infty$ such that
\[
\intO{ \left[ \frac{1}{2} \frac{|\vm|^2}{\vr} + P(\vr) - P'(\Ov{\vr})(\vr - \Ov{\vr}) - P(\Ov{\vr}) \right] (\tau_n, \cdot) } \to 0.
\]
As shown in Theorem \ref{TL1}, the energy functional admits a limit
\[
\intO{ \left[ \frac{1}{2} \frac{|\vm|^2}{\vr} + P(\vr) - P'(\Ov{\vr})(\vr - \Ov{\vr}) - P(\Ov{\vr}) \right] (\tau, \cdot) } \to
\mathcal{E}_\infty \ \mbox{as}\ \tau \to \infty;
\]
whence $\mathcal{E}_\infty = 0$ and the proof of Theorem \ref{TC1} is complete.

\section{Unconditional convergence}
\label{CC}

Theorem \ref{TC1} may seem rather awkward as extra hypotheses are imposed not only on the structural properties of $p$ and
$F$ but also on the solution itself. In this section, we try to remedy the problem at the expense of stronger restrictions on the
exponents $\gamma$ and $q$. As observed in Remark \ref{RC2}, the problem \eqref{P1}--\eqref{P5} admits
a dissipative solution with a vanishing traceless component of $\mathfrak{R}$ as soon as the exponents $\gamma$ and $q$ satisfy
\eqref{A8}. This motivates the following modification of the set $\mathcal{U}[\vr_0, \vm_0]$ that we replace by
\[
\begin{split}
\widetilde{\mathcal{U}}[\vr_0, \vm_0] = \Big\{ &[\vr, \vu, \mathcal{E}] \ \Big|\
[\vr, \vu] \ \mbox{--a dissipative solutions with the initial data}\
[\vr_0, \vm_0] \\ &\mbox{and the associated turbulent energy}\  \mathcal{E},\ \mathfrak{R} - \frac{1}{d} {\rm tr}[\mathfrak{R}] \mathbb{I} =0 \Big\}
\end{split}
\]

\begin{Theorem}[Unconditional convergence to equilibria] \label{TCC}
Let $\Omega \subset \R^d$, $d=2,3$ be a bounded Lipschitz domain. Suppose that $F$ and $p$ comply with the hypotheses \eqref{S1},
\eqref{S3}, \eqref{S5}, and \eqref{A6}. In addition, suppose that the exponents $\gamma$ and $q$ appearing
in \eqref{S2} and \eqref{S6}, respectively, satisfy
\begin{equation} \label{A88}
\frac{1}{\gamma} + \frac{1}{q} \leq 1 \ \mbox{if}\ q > \frac{d}{2},\
\frac{\gamma + 1}{2 \gamma} + \frac{d - q}{dq} < 1
\ \mbox{if}\ q \leq \frac{d}{2}.
\end{equation}
Let $[\vr, \vu]$ be a solution to the problem \eqref{P1}--\eqref{P5} maximal in $\widetilde{\mathcal{U}}[\vr_0, \vm_0]$
in the sense of Definition \ref{SD2}, meaning minimal in $\widetilde{\mathcal{U}}[\vr_0, \vm_0]$ with respect to the relation $\prec$.

Then
\[
\vr \vu (\tau, \cdot) \to 0 \ \mbox{in}\ L^{\frac{2 \gamma}{\gamma + 1}}(\Omega; \R^d),\
\vr(\tau, \cdot) \to \Ov{\vr}\ \mbox{in}\ L^\gamma(\Omega), \ \Ov{\vr} = \frac{1}{|\Omega|} \intO{ \vr_0},
\]
as $\tau \to \infty$.

\end{Theorem}

\begin{Remark} \label{RCC}

It is easy to check that \eqref{A88} implies
\[
\frac{1}{\gamma} + \frac{1}{q} \leq 1 \ \mbox{and} \ \frac{\gamma + 1}{2 \gamma} + \frac{d - q}{dq} < 1.
\]

\end{Remark}

\begin{proof}

First observe that, in view the hypothesis \eqref{A88} and Remark \ref{RC2}, the set
$\widetilde{\mathcal{U}}[\vr_0, \vm_0]$ is non--empty for any finite energy data $[\vr_0, \vm_0]$.
Now it is enough to show that any $[\vr, \vu]$ maximal in $\widetilde{\mathcal{U}}[\vr_0, \vm_0]$ satisfies the hypotheses of Theorem \ref{TC1}.

To begin,
\[
\mathfrak{R} - \frac{1}{d} {\rm tr}[\mathfrak{R}] \mathbb{I} = 0,
\]
in particular the hypothesis \eqref{A9} holds.

Next, repeating the arguments of the proof of Theorem \ref{TL1} we can show that the Reynolds stress $\mathfrak{R}$ associated
to $[\vr, \vu]$ vanishes for $\tau \to \infty$ as required in \eqref{A9a}.

Finally, it follows from \eqref{A88} (cf. Remark
\ref{RCC}) and the DiPerna--Lions theory \cite{DL}, that $[\vr, \vu]$ satisfies the renormalized equation of continuity \eqref{A1}. Thus the solution $[\vr, \vu]$ complies with all hypotheses of Theorem \ref{TC1}, which completes the proof.

\end{proof}

\section{Concluding remarks}
\label{D}

The results presented above can be extended in a straightforward manner to the system driven by a potential external force:
\[
\begin{split}
\partial_t \vr + \Div (\vr \vu) &= 0, \\
\partial_t (\vr \vu) + \Div (\vr \vu \otimes \vu) + \Grad p(\vr) &= \Div \mathbb{S} + \vr \Grad G,\ G = G(x).
\end{split}
\]
Indeed the corresponding energy functional reads
\[
\intO{ \left[ \frac{1}{2} \frac{|\vm|^2}{\vr} + P(\vr) - \vr G \right] },
\]
which can be rewritten as
\[
\intO{ \left[ \frac{1}{2} \frac{|\vm|^2}{\vr} + P(\vr) - P'(\tvr) (\vr - \tvr) - P(\tvr) \right] }
\]
modulo an additive constant. Here $\tvr$ is the associated equilibrium state solving
\begin{equation} \label{D1}
\Grad \tvr = \tvr \Grad G \ \mbox{in}\ \Omega.
\end{equation}
Apparently, equation \eqref{D1} admits infinitely many solution, however uniqueness can be restored for certain potentials
$F$ by prescribing the total mass
\begin{equation} \label{D2}
M = \intO{ \tvr }.
\end{equation}
As shown in \cite{FP7}, \cite{FP9}, the problem \eqref{D1}, \eqref{D2} admits a unique non--negative solution $\tvr$ as soon as the level sets
\[
[ G > k ] = \left\{ x \in \Omega \ \Big|\ G(x) > k \right\}
\]
are connected for any $k$. Under these circumstances, Theorems \ref{TL1}, \ref{TC1}, \ref{TCC} remain valid with obvious modifications in the proof.

\centerline{\bf Acknowledgement}

The paper was written when E.F. was visiting the Dong-A University in Busan. He gladly acknowledges the hospitality and support provided.


\begin{thebibliography}{10}

\bibitem{AbbFeiNov}
A.~Abbatiello, E.~Feireisl, and A.~Novotn{\' y}.
\newblock Generalized solutions to models of compressible viscous fluids.
\newblock {\em Archive Preprint Series}, 2019.
\newblock {\bf arxiv preprint No. 1912.12896}.

\bibitem{BOG}
M.~E. Bogovskii.
\newblock Solution of some vector analysis problems connected with operators
  div and grad (in {R}ussian).
\newblock {\em Trudy Sem. S.L. Sobolev}, {\bf 80}(1):5--40, 1980.

\bibitem{BreFeiHof19}
D.~Breit, E.~Feireisl, and M.~Hofmanov{\' a}.
\newblock Solution semiflow to the isentropic {E}uler system.
\newblock {\em Arxive Preprint Series}, {\bf arXiv 1901.04798}, 2019.
\newblock To appear in Arch. Rational Mech. Anal.

\bibitem{Breni}
Y.~Brenier.
\newblock The initial value problem for the {E}uler equations of incompressible
  fluids viewed as a concave maximization problem.
\newblock {\em Comm. Math. Phys.}, {\bf 364}(2):579--605, 2018.

\bibitem{Daf55}
C.~M. Dafermos.
\newblock Maximal dissipation in equations of evolution.
\newblock {\em J. Differential Equations}, {\bf 252}(1):567--587, 2012.

\bibitem{Daf4}
C.M. Dafermos.
\newblock The second law of thermodynamics and stability.
\newblock {\em Arch. Rational Mech. Anal.}, {\bf 70}:167--179, 1979.

\bibitem{DiP2}
R.J. Di{P}erna.
\newblock Measure-valued solutions to conservation laws.
\newblock {\em Arch. Rat. Mech. Anal.}, {\bf 88}:223--270, 1985.

\bibitem{DL}
R.J. DiPerna and P.-L. Lions.
\newblock Ordinary differential equations, transport theory and {S}obolev
  spaces.
\newblock {\em Invent. Math.}, {\bf 98}:511--547, 1989.

\bibitem{FeNo6A}
E.~Feireisl and A.~Novotn\'y.
\newblock {\em Singular limits in thermodynamics of viscous fluids}.
\newblock Advances in Mathematical Fluid Mechanics. Birkh\"auser/Springer,
  Cham, 2017.
\newblock Second edition.

\bibitem{FNP}
E.~Feireisl, A.~Novotn{\' y}, and H.~Petzeltov{\' a}.
\newblock On the existence of globally defined weak solutions to the
  {N}avier-{S}tokes equations of compressible isentropic fluids.
\newblock {\em J. Math. Fluid Mech.}, {\bf 3}:358--392, 2001.

\bibitem{FP7}
E.~Feireisl and H.~Petzeltov{\'a}.
\newblock On the zero-velocity-limit solutions to the {N}avier-{S}tokes
  equations of compressible flow.
\newblock {\em Manuscr. Math.}, {\bf 97}:109--116, 1998.

\bibitem{FP9}
E.~Feireisl and H.~Petzeltov{\'a}.
\newblock Large-time behaviour of solutions to the {N}avier-{S}tokes equations
  of compressible flow.
\newblock {\em Arch. Rational Mech. Anal.}, {\bf 150}:77--96, 1999.

\bibitem{GAL}
G.~P. Galdi.
\newblock {\em An introduction to the mathematical theory of the {N}avier -
  {S}tokes equations, I.}
\newblock Springer-Verlag, New York, 1994.

\bibitem{GEHEHI}
M.~Gei{\ss}ert, H.~Heck, and M.~Hieber.
\newblock On the equation {${\rm div}\,u=g$} and {B}ogovski\u\i's operator in
  {S}obolev spaces of negative order.
\newblock In {\em Partial differential equations and functional analysis},
  volume 168 of {\em Oper. Theory Adv. Appl.}, pages 113--121. Birkh\"auser,
  Basel, 2006.

\bibitem{LI}
P.-L. Lions.
\newblock {\em Mathematical topics in fluid dynamics, Vol.1, Incompressible
  models}.
\newblock Oxford Science Publication, Oxford, 1996.

\bibitem{LI4}
P.-L. Lions.
\newblock {\em Mathematical topics in fluid dynamics, Vol.2, Compressible
  models}.
\newblock Oxford Science Publication, Oxford, 1998.

\bibitem{MNRR}
J.~M{\' a}lek, J.~Ne{\v c}as, M.~Rokyta, and M.~R{\accent23u}{\v z}i{\v c}ka.
\newblock {\em Weak and measure-valued solutions to evolutionary PDE's}.
\newblock Chapman and Hall, London, 1996.

\bibitem{MAM1}
A.~E. Mamontov.
\newblock Global solvability of the multidimensional {N}avier- {S}tokes
  equations of a compressible fluid with nonlinear viscosity, {I}.
\newblock {\em Siberian Math. J.}, {\bf 40}(2):351--362, 1999.

\bibitem{MAM2}
A.~E. Mamontov.
\newblock Global solvability of the multidimensional {N}avier- {S}tokes
  equations of a compressible fluid with nonlinear viscosity, {I}{I}.
\newblock {\em Siberian Math. J.}, {\bf 40}(3):541--555, 1999.

\end{thebibliography}

\def\cprime{$'$} \def\ocirc#1{\ifmmode\setbox0=\hbox{$#1$}\dimen0=\ht0
  \advance\dimen0 by1pt\rlap{\hbox to\wd0{\hss\raise\dimen0
  \hbox{\hskip.2em$\scriptscriptstyle\circ$}\hss}}#1\else {\accent"17 #1}\fi}

\end{document}